\documentclass[a4paper,11pt]{amsart}
\usepackage[colorlinks, linkcolor=blue,anchorcolor=Periwinkle,
    citecolor=blue,urlcolor=Emerald]{hyperref}
\usepackage[all]{xy}
\SelectTips{cm}{}

\usepackage{graphicx}
\usepackage{psfrag}

\usepackage{tikz}
\usepackage{tikz-cd}

\newcommand{\cf}{{\em cf.}\ }

\newcommand{\ko}{\: , \;}

\newcommand{\ol}[1]{\overline{#1}}

\numberwithin{equation}{section}
\newtheorem{theorem}[subsection]{Theorem}
\newtheorem{classification-theorem}[subsection]{Classification Theorem}
\newtheorem{decomposition-theorem}[subsection]{Decomposition Theorem}
\newtheorem{proposition-definition}[subsection]{Proposition-Definition}
\newtheorem{periodicity-conjecture}[subsection]{Periodicity Conjecture}

\newtheorem{proposition}[subsection]{Proposition}

\newtheorem{conjecture}[subsection]{Conjecture}

\newtheorem{remark}[subsection]{Remark}

\renewcommand{\mod}{\mathrm{mod}\,}

\newcommand{\Hom}{\mathrm{Hom}}

\newcommand{\Ext}{\mathrm{Ext}}

\newcommand{\C}{\mathbb{C}}

\newcommand{\A}{\mathbb{A}}
\renewcommand{\P}{\mathbb{P}}

\newcommand{\iso}{\xrightarrow{_\sim}}

\newcommand{\End}{\mathrm{End}}

\newcommand{\ca}{{\mathcal A}}

\newcommand{\cc}{{\mathcal C}}

\newcommand{\cF}{{\mathcal F}}

\newcommand{\cR}{{\mathcal R}}

\newcommand{\ct}{{\mathcal T}}
\newcommand{\cu}{{\mathcal U}}

\renewcommand{\phi}{\varphi}

\newcommand{\bt}{\bullet}

\renewcommand{\tilde}[1]{\widetilde{#1}}

\begin{document}

\date{April 14, 2022}

\title{On Leclerc's conjectural cluster structures for open Richardson varieties}

\author{Peigen Cao}
\address{
     Universit\'e Paris Cit\'e,
    UFR de Math\'ematiques,
    CNRS,
   Institut de Math\'ematiques de Jussieu--Paris Rive Gauche, IMJ-PRG,
    B\^{a}timent Sophie Germain,
    75205 Paris Cedex 13,
    France \newline %
Present address: Einstein Institute of Mathematics, Edmond J. Safra Campus, The Hebrew University of Jerusalem, Jerusalem
91904, Israel
    }
    
    \email{peigencao@126.com}

\author{Bernhard Keller}
\address{
Universit\'e Paris Cit\'e, Sorbonne Universit\'e, CNRS IMJ--PRG \\
8 place Aur\'elie Nemours, 75013 Paris, France}
\email{bernhard.keller@imj-prg.fr}
\urladdr{https://webusers.imj-prg.fr/~bernhard.keller/}

\begin{abstract} In 2016, Leclerc constructed conjectural cluster structures on 
open Richardson varieties using representations of preprojective algebras.
A variant with more explicit seeds was obtained by M\'enard in his thesis.
We show that M\'enard's seeds do yield {\em upper} cluster algebra structures
on open Richardson varieties and discuss the problems that remain 
in order to prove that they are cluster algebra structures.
\end{abstract}




\maketitle

\section{Introduction}
\label{s:Introduction}

Open Richardson varieties were introduced by Kazhdan--Lusztig
\cite{KazhdanLusztig88}. They
are relevant for Kazhdan--Lusztig polynomials
\cite{KazhdanLusztig88, Deodhar85}, 
the study of total positivity in the Grassmannian 
\cite{Lusztig98pos, Rietsch99, MarshRietsch04, RietschWilliams08}, 
the Poisson geometry of the flag variety
\cite{GoodearlYakimov09} and many other subjects.

It is natural to ask whether open Richardson varieties carry cluster structures
compatible with total positivity and Poisson geometry. In 2016, Leclerc
\cite{Leclerc16} gave a conjectural positive answer using representations
of preprojective algebras. His conjecture was slightly modified and made
more explicit by M\'enard \cite{Menard21, Menard22}. In this note, 
based on \cite{Keller21}, we show that M\'enard's seed does provide 
an {\em upper} cluster algebra structure on each open Richardson variety.
We also discuss the problems that remain 
in order to prove that it is a cluster algebra structure.

In type A, a (possibly different) upper cluster algebra structure was obtained
using completely different methods by Gracie Ingermanson
in her thesis \cite{Ingermanson19} under the supervision
of David Speyer. Open Richardson varieties are special cases of braid varieties.
In this more general framework, much stronger results will be contained
in the forthcoming work of two groups of mathematicians:
\begin{itemize}
\item[-] Roger Casals, Eugene Gorsky, Mikhail Gorsky, Ian Le, Linhui Shen,
and Jos\'e Simental in \cite{CGGLSS22} and
\item[-] Pavel Galashin, Thomas Lam and Melissa Sherman-Bennett
in \cite{GalashinLamShermanBennett22a, GalashinLamShermanBennett22b}, 
cf.~also \cite{GalashinLam22}.
\end{itemize}

\subsection*{Acknowledgments} B.~Keller thanks Bernard Leclerc
and Etienne M\'enard for stimulating discussions. P.~Cao is supported by the ERC Grant No. 669655.
We are grateful to the authors of \cite{CGGLSS22} for providing us with a preliminary 
version of their preprint.

\section{Open Richardson varieties}
\label{s:Open Richardson varieties}

Let $\Delta$ be a simply laced Dynkin diagram, for example the diagram $A_4$
given by a chain of length $4$.
Let $G$ be the associated simple, simply connected complex algebraic group,
for example $Sl_5(\C)$. Let $B, B^- \subseteq G$ be opposite Borel subgroups,
for example the subgroups of upper/lower triangular matrices in $Sl_5(\C)$.
Let $H=B\cap B^-$ be the associated maximal torus and $W=N_B(H)/H$ the
Weyl group, for example the subgroup $H\subseteq Sl_5(\C)$ of diagonal
matrices and the symmetric group $S_5$. Let $X=B^-\setminus G$ be the flag variety
and $\pi: G \to X$ the canonical projection, for example the variety of
complete flags
\[
0=F_0 \subset F_1 \subset \cdots \subset F_5=\C^5
\]
in the space of rows $\C^5$ with its natural right action by $Sl_5(\C)$. 
We have the Schubert decomposition
\[
X = \coprod_{w\in W} C_w
\]
into the Schubert cells $C_w = \pi(wB^-)$, which are affine spaces of
dimension equal to the length $l(w)$. Dually, we have the opposite Schubert
decomposition
\[
X=\coprod_{v\in W} C^v
\]
into the opposite Schubert cells $C^v=\pi(vB)$, which are affine spaces
of dimension equal to $l(w_0)-l(v)$, where $w_0$ is the longest element of $W$.
For a pair of Weyl group elements $(v,w)$, the {\em open Richardson variety
$\cR_{v,w}$} is defined as the intersection $C^v\cap C_w$. It is non-empty
if and only if $v\leq w$ for the Bruhat order and in this case, it is a smooth,
irreducible, locally closed subvariety of $C_w$ of dimension $l(w)-l(v)$
which is affine but not an affine space, in general.
In the minimal example where $G=Sl_2(\C)$, the flag variety $X=B^-\setminus G$
identifies with the projective line $\P^1(\C)$ by the map sending a
$2\times 2$-matrix to the line generated by its first row. Under this identification, if $s$ generates the Weyl group $S_2$, we have
\[
\xymatrix@R=0cm{
C_e=\{0\}  & C_s = \P^1(\C)\setminus \{0\} & \\
C^e=\A^1(\C) &  C^s=\{\infty\} &  \\
\cR_{e,e}=\{0\} &  \cR_{e,s} = \P^1(\C)\setminus \{0,\infty\} & \cR_{s,s}={\infty}.
}
\]
We refer to section~\ref{s:Introduction} for the most relevant
references on open Richardson varieties. 
It is natural to ask whether they carry cluster structures
compatible with total positivity and Poisson geometry. In 2016, Leclerc
\cite{Leclerc16} gave a conjectural positive answer using representations
of preprojective algebras. There is a direct link between such representations
and the coordinate algebra $\C[N]$ of the unipotent radical $N$ of $B$.
We will recall it in the next section. In turn, the coordinate algebra $\C[N]$
is linked to the coordinate algebra $\C[\cR_{v,w}]$ of the affine
variety $\cR_{v,w}$ as follows: Put $N(v)=N\cap v^{-1}N^-v$, where
$N^-$ is the unipotent radical of $B^-$, and $N'(w)=N\cap w^{-1}Nw$.
 Let $\C[N]^{v,w}\subseteq \C[N]$ be the subalgebra of double invariants
\[
\C[N]^{v,w} = \mbox{}^{N(v)}\C[N]^{N'(w)}.
\]
Let $M_{v,w}$ be the multiplicative subset of $\C[N]^{v,w}$ generated
by the irreducible factors of
\[
D_{v,w} = \prod_{i\in I} \Delta_{v^{-1}(\varpi_i), w^{-1}(\varpi_i)} \ko
\]
where $I$ is the set of vertices of the Dynkin diagram and
the factors of the product are generalized minors, \cf section~2.2
of \cite{Leclerc16}. In section~2.8 of \cite{Leclerc16}, Leclerc
constructs an algebra isomorphism
\[
\C[N]^{v,w}[M_{v,w}^{-1}] \iso \C[\cR_{v,w}].
\]
We thus obtain the following diagram summing up the relations
between the coordinate algebras $\C[N]$ and $\C[\cR_{v,w}]$
\[
\begin{tikzcd}
    \C[N]^{v,w} \arrow[r,hook] \arrow[d, "\mbox{\small can}"] & \C[N] &  \\
    \C[N]^{v,w}[M_{v,w}^{-1}]\arrow[r,"\sim"]  & \C[\cR_{v,w}].
\end{tikzcd}
\]

\section{Additive categorification and Leclerc's conjecture}
\label{s:Additive categorification}

\subsection{The case of $N \iso C_{w_0}$} \label{ss:case of N}
We keep the notations and assumptions of section~\ref{s:Open Richardson varieties}.
Let $\Lambda$ be the preprojective algebra of $\Delta$ over $k=\C$. 
For example, if $\Delta$ is the diagram $A_4$
\[
\begin{tikzcd} 1 \ar[r, no head] & 2 \ar[r, no head] & 3 \ar[r, no head] & 4
\end{tikzcd} \ko
\]
then, up to isomorphism, $\Lambda$ is the $k$-algebra presented by the quiver
\[
\begin{tikzcd}
1 \ar[r, shift left=1ex, "\alpha"] & 
2 \ar[l, shift left=1ex, "\alpha^*"]  \ar[r, shift left=1ex, "\beta"] &
3 \ar[l, shift left=1ex, "\beta^*"]  \ar[r, shift left=1ex, "\gamma"] &
4 \ar[l, shift left=1ex, "\gamma^*"]
\end{tikzcd}
\]
with relations
\[
-\alpha^* \alpha=0 \ko \alpha\alpha^* - \beta^* \beta =0 \ko \beta\beta^* - \gamma^* \gamma=0 \ko
\gamma \gamma^*=0.
\]
Let us recall some important properties of $\Lambda$ and the category $\mod \Lambda$ of
$k$-finite-dimensional (right) $\Lambda$-modules:
\begin{itemize}
\item[a)] The algebra $\Lambda$ is finite-dimensional and selfinjective so that the category
$\mod\Lambda$ becomes a Frobenius category. 
\item[b)] As shown by Crawley-Boevey \cite{CrawleyBoevey00},
for finite-dimensional $\Lambda$-modules $L$ and $M$, we have a bifunctorial
isomorphism
\[
D\Ext^1_\Lambda(L,M) \iso \Ext_\Lambda(M,L) \ko
\]
where $D=\Hom_k(?,k)$ is the duality over the ground field. This means that
the Frobenius category $\mod\Lambda$ is {\em stably $2$-Calabi--Yau}. 
\item[c)] The category $\mod\Lambda$ contains (basic) {\em cluster-tilting objects}
\[
T=T_1 \oplus  \cdots \oplus T_m \ko
\]
where the $T_i$ are indecomposable (and pairwise non isomorphic) and $m$ is the length of the longest element $w_0$. These can
be {\em mutated} at each non projective summand $T_i$.
\item[d)] Each reduced expression $\ol{w_0}$ of the longest element $w_0$ yields
a canonical cluster-tilting object $T_{\ol{w_0}}$ which, up to mutation, is independent
of the choice of $\ol{w_0}$.
\item[e)] We have a canonical {\em cluster-character}
\[
\phi: \mod\Lambda \to \C[N]
\]
constructed by Geiss--Leclerc--Schr\"oer \cite{GeissLeclercSchroeer06} using 
work of Lusztig \cite{Lusztig00}.
\end{itemize}
For an ice quiver $Q$, let us write $\ca^+(Q)$ for the cluster algebra with {\em non
invertible coefficients} associated with $Q$ and $\ca(Q)$ for the cluster algebra
with {\em invertible coefficients} associated with $Q$. We always denote the
initial cluster variables by $x_i$, $i\in Q_0$.

\begin{theorem}[\cite{GeissLeclercSchroeer06}] If
$
T=T_1 \oplus  \cdots \oplus T_m 
$
is a basic cluster-tilting object mutation-equivalent to $T_{\ol{w_0}}$ and
$Q(T)$ is the ice quiver of its endomorphism algebra $\End_\Lambda(T)$ with
frozen vertices corresponding to the projective-injective indecomposable summands of $T$,
then $(Q(T), (\phi(T_i)))$ is an initial seed for a cluster structure on $\C[N]$.
Moreover, the algebra isomorphism morphism taking $x_i$ to $\phi(T_i)$ fits into a commutative
diagram
\[
\begin{tikzcd}
\ca^+(Q(T)) \arrow[d] \arrow[r, "\sim"] & \C[N] \arrow[d] \\
\ca(Q(T)) \arrow[r,"\sim"] & \C[\cR_{e,w_0}].
\end{tikzcd}
\]
\end{theorem}
For example, suppose that $\Delta=A_4$ and 
\[
\ol{w_0}=s_1 s_2 s_3 s_4 s_1 s_2 s_3 s_1 s_2 s_1.
\]
We refer to \cite{GeissLeclercSchroeer07d} for the construction of
the canonical cluster-tilting object $T=T_{\ol{w_0}}$ and the computation
of the corresponding ice quiver $Q(T)$ given by
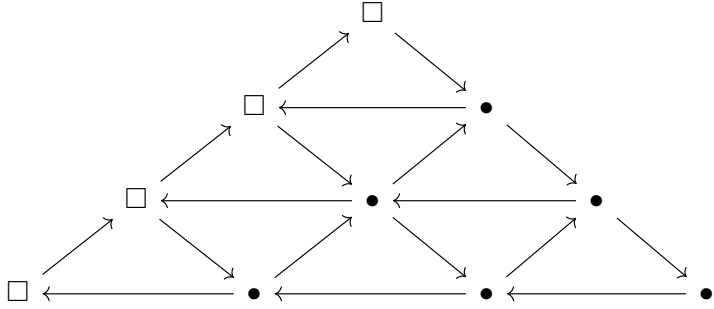
\[
\begin{tikzcd}
 & & & \square \ar[rd] & & & \\
 & & \square \ar[ru] \ar[rd] & & \bt \ar[ll] \ar[rd] & & \\
 & \square \ar[ru] \ar[rd] & & \bt \ar[ru] \ar[ll] \ar[rd] & & \bt \ar[ll] \ar[rd] &  \\
 \square \ar[ru] & & \bt \ar[ll] \ar[ru] & & \bt \ar[ll] \ar[ru] & & \bt \ar[ll]
 \end{tikzcd}
\]
where the squares denote frozen vertices.
In this case, the subgroup $N$ is formed by the upper unitriangular
$4\times 4$-matrices and the isomorphism
\[
\ca^+(Q(T)) \iso \C[N]
\]
sends the $x_i$ to certain maximal minors.

\subsection{Case of $C_w$} \label{ss:Case of C_w}
Recall that a {\em torsion pair} in $\mod\Lambda$ is a pair $(\ct,\cF)$ of
strictly full subcategories such that we have
\begin{itemize}
\item[a)] $\Hom(\ct,\cF)=0$ and
\item[b)] for each $M\in\mod\Lambda$, there is a short exact sequence
\[
\begin{tikzcd}
0 \ar[r] & M_\ct \ar[r] & M \ar[r] & M^\cF \ar[r] & 0 \ko
\end{tikzcd}
\]
where $M_\ct$ belongs to $\ct$ and $M^\cF$ to $\cF$.
\end{itemize}
Here the submodule $M_\ct$ is called the {\em torsion part} and
$M^\cF$ the {\em torsion-free part} of $M$. The subcategory
$\ct$ is called a {\em torsion class} and $\cF$ a {\em torsion-free class}.
Torsion classes ordered by inclusion form a poset. 

For two elements $v$ and $w$ of the Weyl group $W$, we write
$v \leq_R w$ if $w$ admits a reduced expression $\ol{w}$ equal
to the concatenation $\ol{v}\ol{x}$ of a reduced expression $\ol{v}$
for $v$ with a reduced word $\ol{x}$. The relation $\leq_R$ is
called the {\em weak right order} on $W$. Clearly, the relation
$v\leq_R w$ implies that $v\leq w$ in the Bruhat order but the
converse does not hold in general.

Recall the a subcategory $\cc$ of $\mod\Lambda$ is {\em functorially finite}
if, for each $M\in \mod\Lambda$, there are morphisms $C_0 \to M \to C^0$
with $C_0, C^0\in\cc$ such that for each $C\in \cc$, each morphism $C \to M$ 
factors through $C_0 \to M$ and each morphism $M \to C$ factors through
$M \to C^0$.

\begin{theorem}[\cite{Mizuno14}] We have a canonical isomorphism
of posets $w \mapsto \cc_w$ from $(W,\leq_R)$ to the poset of
functorially finite torsion classes of $\mod\Lambda$.
\end{theorem}

\begin{theorem}[\cite{DemonetIyamaJasso19}] If $A$ is a finite-dimensional
$k$-algebra such that $\mod A$ admits only finitely many functorially finite torsion classes,
then each torsion class is functorially finite.
\end{theorem}

By combining the two theorems, we see that all torsion classes of
$\mod\Lambda$ are functorially finite and that they are canonically parametrized
by the elements of $W$ via the bijection $w \mapsto \cc_w$.

\begin{theorem}[\cite{BuanIyamaReitenScott09}]  \label{thm:BIRS09}
If $\cc\subseteq \mod\Lambda$ is
an extension-closed, functorially finite full subcategory, it becomes a Frobenius category
(for the exact structure inherited from $\mod\Lambda$) which is stably $2$-Calabi--Yau
and has a cluster structure. In particular, the category $\cc$ contains a cluster-tilting object.
\end{theorem}

For an ice quiver $Q$, let we write $\cu^+(Q)$ for the upper cluster algebra with
{\em non invertible} coefficients associated with $Q$.

We say that an ice quiver $Q$ has a {\em reddening sequence} if the non frozen part of $Q$ has a reddening 
sequence in the sense of \cite{Keller11c}.

\begin{theorem}[\cite{GeissLeclercSchroeer07d, BuanIyamaReitenScott09}]
Fix $w\in W$ and let $\ol{w}$ be a reduced expression for $w$.
\begin{itemize}
\item[a)] There is a canonical cluster-tilting object $T_{\ol{w}}$ of $\cc_w$ which, up to
mutation, only depends on $w$.
\item[b)] The ice quiver $Q_{\ol{w}}$ of the endomorphism algebra of $T_{\ol{w}}$ has
an explicit description (up to the determination of the frozen subquiver).
\item[c)] The ice quiver $Q_{\ol{w}}$ has a reddening sequence \cite{Keller11c} and
we have $\ca^+(Q_{\ol{w}})= \cu^+(Q_{\ol{w}})$.
\end{itemize}
\end{theorem}

As an example, consider $\Delta=A_4$ and 
$\ol{w}=s_1 s_2 s_3 s_1 s_2 s_4 s_3$. Then the quiver $Q_{\ol{w}}$ is given by
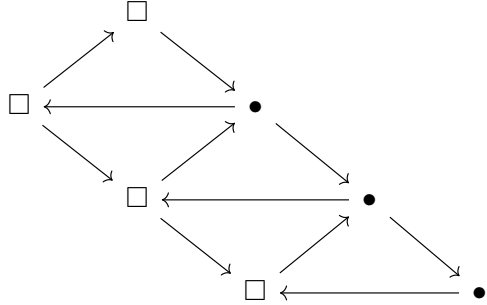
\[
\begin{tikzcd}
& \square \ar[rd] & & & \\
\square \ar[rd] \ar[ru] & & \bt \ar[rd] \ar[ll] & & \\
 & \square \ar[ru] \ar[rd] & & \bt \ar[rd] \ar[ll] & \\
 & & \square \ar[ru] & & \bt \ar[ll]
 \end{tikzcd}
 \]
 Let $w\in W$ and let $\ol{w}$ be a reduced expression for $w$. 
 Recall that we have defined
 $N(w)=N\cap w^{-1}N^-w$, where $N^-$ is the unipotent radical of $B^-$, and 
 $N'(w)=N\cap w^{-1}Nw$. We have a canonical isomorphism
 \[
 \C[N(w)] \iso \C[N]^{N'(w)}.
 \]
 
 \begin{theorem}[\cite{GeissLeclercSchroeer07d, GeissLeclercSchroeer11b}] Choose a
 decomposition into indecomposables
 \[
 T_{\ol{w}} = T_1 \oplus \cdots \oplus T_{l(w)}.
 \]
 Then the pair $(Q_{\ol{w}}, (\phi(T_i)))$ is an initial seed for a cluster structure
 on $\C[N(w)]$, i.e. the algebra morphism taking $x_i$ to $\phi(T_i)$ is
 an isomorphism. Moreover, it makes the following square commutative
 \[
 \begin{tikzcd}
 \ca^+(Q_{\ol{w}}) \ar[d] \ar[r,"\sim"] & \C[N(w)] \ar[d] \ar[r,"\sim"] & \C[N]^{N'(w)} \ar[dl] \\
 \ca(Q_{\ol{w}}) \ar[r,"\sim"] & \C[\cR_{e,w}] & 
 \end{tikzcd}
 \]
 \end{theorem}
 
 \subsection{The case of $\cR_{v,w}$} We follow \cite{Leclerc16}. Let $v\leq w$ be elements of $W$.
 Let $(\cc_v,\cc^v)$ be the torsion pair associated with $v$. Define
 \[
 \cc_{v,w} = \cc^v \cap \cc_w \subseteq \mod\Lambda.
 \]
 Clearly, this subcategory is extension-closed. By the results of Auslander--Smal\o\ 
 \cite{AuslanderSmaloe81}, it is also functorially finite in $\mod\Lambda$. 
 By Theorem~\ref{thm:BIRS09}, the category $\cc_{v,w}$ has a cluster structure.
 If $T$ is a cluster-tilting object of $\cc_w$, then its $\cc^v$-torsion-free part $T^{\cc^v}$
 is a cluster-tilting object of $\cc_{v,w}$, by Proposition~3.12 of \cite{Leclerc16}. 
 However, in general, the module $T^{\cc^v}$ is not basic. Let $T_{v, \ol{w}}$ be
 a maximal basic summand of the cluster-tilting object $T_{\ol{w}}^{\cc^v}$ 
 and let $Q_{v,\ol{w}}$ be the quiver of the endomorphism algebra of $T_{v,\ol{w}}$.
 
 \begin{theorem}[Leclerc \cite{Leclerc16}] \label{thm:lec16}
 \begin{itemize} 
 \item[a)] The $\C$-span of $\phi(\cc_{v,w})\subseteq \C[N]$ equals
 \[
 \C[N]^{v,w} = \mbox{}^{N(v)} \C[N]^{N'(w)}.
 \]
 \item[b)] The map $\phi: \cc_{v,w} \to \C[N]$ induces injective algebra morphisms
 \[
 \begin{tikzcd}
 \ca^+(Q_{v,\ol{w}}) \ar[d] \ar[r, hook, "\tilde{\phi}"] & \C[N]^{v,w} \ar[d] \\
 \ca(Q_{v,\ol{w}}) \ar[r,hook,"\tilde{\phi}_{loc}"'] & \C[\cR_{v,w}]
 \end{tikzcd}
 \]
 and $\dim \cR_{v,w}$ equals the number of vertices of $Q_{v,\ol{w}}$.
 \item[c)] The algebra embedding $\tilde{\phi}_{loc}$ is an isomorphism if $v\leq_R w$ or if
 $\cc_{v,w}$ has only finitely many indecomposables (up to isomorphism).
 \end{itemize}
 \end{theorem}
 
 \begin{conjecture}[Leclerc \cite{Leclerc16}] The algebra embedding $\tilde{\phi}_{loc}$ is
 always an isomorphism.
 \end{conjecture}
 
 One difficulty arises from the fact that Leclerc's seed $(Q_{v,\ol{w}}, (\phi(T_i)))$
 is not known explicitly. The following theorem is the first to have made it explicit in certain cases.
 
 \begin{theorem}[Serhiyenko--Sherman-Bennett--Williams \cite{SerhiyenkoShermanBennettWilliams20}] 
 Leclerc's seed equals the canonical seed defined by a plabic graph if
 $\cR_{v,w}$ is an open Schubert variety (in type $A$).
 \end{theorem}
 
 This theorem implies Leclerc's conjecture for these cases since we have $v\leq_R w$
 if $\cR_{v,w}$ is an open Schubert variety.
 
 \begin{theorem}[Galashin--Lam \cite{GalashinLam19}] 
 Leclerc's seed seed equals the canonical seed defined
 by a plabic graph if $\cR_{v,w}$ is an open positroid variety (i.e. a type $A$ open
 Richardson variety in the Grassmannian). Moreover, the conjecture holds in
 this case.
 \end{theorem}
 
 Notice that in the situation of the theorem, we may have $v \not\leq_R w$. 
 
 \section{M\'enard's results}
 
 We keep the notations and assumptions of the preceding section.
 
 \begin{theorem}[M\'enard \cite{Menard22}] There is an explicit sequence of mutations
 transforming $T_{\ol{w}}$ into a cluster-tilting object $T'_{\ol{w}}$ such that any maximal
 direct summand $M_{v,\ol{w}}$ of $T'_{\ol{w}}$ lying in $\cc_{v,w}$ is a cluster-tilting
 object of $\cc_{v,w}$.
 \end{theorem}
 
 The sequence of mutations in the theorem was conjectured by Jan Schr\"oer. 
 The cluster-tilting object $M_{v,\ol{w}}$ is expected to be isomorphic to $T_{v,\ol{w}}$.
 It yields a (possibly new) candidate seed for $\C[N]^{v,w}$ and $\C[\cR_{v,w}]$.
 In his thesis \cite{Menard21, Menard22}, M\'enard has developed an algorithm
 allowing to explicitly compute the seed associated with $M_{v,\ol{w}}$. 
 It follows from his theorem above that the quiver $Q(M_{v,\ol{w}})$ is a
 {\em cluster reduction} of $Q(T_{\ol{w}})$, i.e. it is obtained from $Q(T_{\ol{w}})$
 by mutating, freezing vertices and deleting certain frozen vertices (in this order). 
 By a theorem of Muller \cite{Muller16}, the existence of reddening sequences
 is preserved under cluster reduction. Since the existence of a reddening sequence
 for the ice quiver $Q(T_{\ol{w}})$ is known \cite{GeissLeclercSchroeer11b}, it follows that the ice quiver
 $Q(M_{v,\ol{w}})$ has a reddening sequence. Moreover, the exchange matrix
 associated with $Q(M_{v,\ol{w}})$ has full rank (this follows 
 essentially from \cite{BuanIyamaReitenScott09}). Thus, the
 upper cluster algebra $\cu(Q(M_{v,\ol{w}}))$ (with invertible coefficients) admits a theta
 basis in the sense of Gross--Hacking--Keel--Kontsevich 
 \cite{GrossHackingKeelKontsevich18} and also a generic basis,
 as shown by Qin \cite{Qin19}. In particular, the image of the
 Caldero--Chapoton map spans the upper cluster algebra
 $\cu(Q(M_{v,\ol{w}}))$ over the algebra of Laurent polynomials
 in the coefficients.

 \section{Upper cluster structure}
 
 We keep the notations and assumptions of the preceding section. Let
 $T$ be M\'enard's cluster-tilting object $M_{v,\ol{w}}$.
 Let $\phi: \ca(Q(T)) \to \C[N]^{v,w}$ be the algebra morphism associated
 with the seed $(Q(T), (\phi(T_i)))$.

 \begin{theorem} The map $\phi$ yields a commutative square
 \[
 \begin{tikzcd}
 \cu^+(Q(T)) \ar[d] \ar[r, "\tilde{\phi}"'] & \C[N]^{v,w} \ar[d]\\
 \cu(Q(T)) \ar[r,"\sim","\tilde{\phi}_{loc}"'] & \C[\cR_{v,w}].
 \end{tikzcd}
 \]
 where the bottom map $\tilde{\phi}_{loc}$ is an isomorphism.
 \end{theorem}
 
\begin{remark} We do not know whether the statement of the
theorem also holds if $T$ is Leclerc's cluster tilting object $T_{v,\ol{w}}$.
\end{remark}
 
\begin{proof}
Since the $\phi(T_i)$ are algebraically
independent, the map $x_i \mapsto \phi(T_i)$ defines a
field embedding
\[
\begin{tikzcd}
\C(x_i) \ar[rr,hook, "\tilde{\phi}"] && \C(N).
\end{tikzcd}
\]
 Let $CC: \cc_{v,w} \to \C(x_i)$ denote the cluster
character associated with the cluster-tilting object $T\in \cc_{v,w}$
in \cite{FuKeller10}.
By Theorem~4 of \cite{GeissLeclercSchroeer12}, the triangle
\[
\begin{tikzcd} 
\C(x_i) \ar[rr,hook, , "\tilde{\phi}"] && \C(N) \\
 & \cc_{v,w} \ar[ul, "CC"] \ar[ur, "\phi"'] &  
 \end{tikzcd}
 \]
 commutes. Now by definition, the map $CC: \cc_{v,w} \to \C(x_i)$
 actually takes its values in $\C[x_i^\pm]$ and by Theorem~1.1 of
 \cite{Plamondon13}, it even takes its values in the upper cluster
 algebra $\cu^+ = \cu^+(Q(T))$ with non invertible coefficients.
 Clearly, the field embedding 
 \[
 \tilde{\phi}: \C(x_i) \to \C(N)
 \]
 induces an isomorphism 
 \[
 \cu^+ \iso \tilde{\phi}(\cu^+) \subset \C(N)
 \]
 and we have the commutative square
 \[
 \begin{tikzcd}
 \cc_{v,w} \ar[d,"CC"'] \ar[r,"\phi"] & \C(N) \\
 \cu^+ \ar[r,"\sim"] & \tilde{\phi}(\cu^+). \ar[u,hook]
 \end{tikzcd}
 \]
 By Theorem \ref{thm:lec16} (a), the $\C$-span of $\phi(\cc_{v,w})$
 equals $\C[N]^{v,w}$. This implies the inclusion
 \[
 \C[N]^{v,w} \subseteq \tilde{\phi}(\cu^+).
 \]
 Since $Q(T)$
 is of full rank, the upper cluster algebra with non invertible
 coefficients is a {\em finite} intersection of Laurent
 polynomial rings (by Cor.~1.9 of \cite{BerensteinFominZelevinsky05},
 the `starfish lemma'). Therefore, the
 upper cluster algebra with invertible coefficients $\cu$
 is the localization of $\cu^+$ at the coefficients. 
 Therefore, from the above inclusion, we deduce that we have
 \[
 \C[N]^{v,w}[M_{v,w}^{-1}] \subseteq \tilde{\phi}(\cu).
 \]
 Here, the symbol $M_{v,w}$ denotes the multiplicative set in $\C[N]^{v,w}$
 introduced at the end of  section~\ref{s:Open Richardson varieties}.

 Since $T$ is M\'enard's cluster-tilting object, we know that the non frozen part of the ice quiver $Q(T)$ has a reddening sequence. 
 By \cite[Theorem 1.2.3]{Qin19}, the upper cluster algebra $\mathcal U$ has a generic basis. This implies that $CC(\cc_{v,w})$ contains a generating set for the $\C[M_{v,w}^{\pm 1}]$-algebra
 $\cu$. Since the $\C$-span of $\phi(\cc_{v,w})$
 equals $\C[N]^{v,w}$, we have the reverse inclusion
 \[
 \tilde{\phi}(\cu) \subseteq \C[N]^{v,w}[M_{v,w}^{-1}].
 \]
 Thus, we obtain the equality 
 \[
 \tilde{\phi}(\cu) =  \C[N]^{v,w}[M_{v,w}^{-1}] = \C[\cR_{v,w}].
 \]
 This is what we had to prove. 
\end{proof}

\section{Towards a cluster structure}

Our hope is that for M\'enard's cluster-tilting object $M_{v,\ol{w}}$,
we have $\ca=\cu$ for the corresponding cluster and upper cluster
algebra with invertible coefficients. Recall that by M\'enard's theorem,
the ice quiver $Q(M_{v,\ol{w}})$ is obtained from $Q(T_{\ol{w}})$ by
cluster reduction, i.e. by mutation, freezing and deletion of 
frozen vertices (in this order). 

\begin{theorem}[Geiss--Leclerc--Schr\"oer \cite{GeissLeclercSchroeer11b}]
\begin{itemize}
\item[a)] We have $\ca=\cu$ for $Q(T_{\ol{w}})$.
\item[b)] The ice quiver $Q(T_{\ol{w}})$ admits a reddening sequence.
\end{itemize}
\end{theorem}

The second property is preserved under cluster reduction by
Muller's theorem \cite{Muller16}. However, it is not clear 
under which conditions this holds for the first property.

\subsection{Preservation of $\cu=\ca$ under freezing?}
Let $Q$ be an ice quiver and $Q'$ the quiver obtained
from $Q$ by freezing the cluster variable $x$ associated
with a non frozen vertex. We then have the algebra inclusions
\[
\ca'\subseteq \ca[x^{-1}] \subseteq \cu[x^{-1}] \subseteq \cu' \ko
\]
where $\ca=\ca(Q)$, \ldots\ . Following Muller \cite{Muller14}, we define
$\ca'$ to be a {\em cluster localization} of $\ca$ at $x$ if $\ca'=\ca[x^{-1}]$.
Unfortunately, it is not clear whether the freezing occuring in 
the passage from Geiss--Leclerc--Schr\"oer's seed for $C_w$
to M\'enard's for $\cR_{v,w}$ is a composition of cluster localizations.
The following theorem may nevertheless be useful.

\begin{theorem} Suppose the exchange matrix associated with the ice
quiver $Q$ is of full rank. Let $\ca$ and
$\cu$ be the associated cluster and upper cluster algebra. If we have $\ca=\cu$
and $\ca'$ is a cluster localization of $\ca$ at $x$, then we have $\ca'=\cu'$.
\end{theorem}

\begin{proof}
Since  the exchange matrix $B$ associated with the ice quiver $Q$ is of full rank and $Q^\prime$ is obtained from $Q$ by freezing the non frozen vertex of $Q$ labeled by $x$, we know that the exchange matrix $B^\prime$ associated with the ice quiver $Q^\prime$ is also of full rank. Hence, the starfish lemma \cite[Corollary 1.9]{BerensteinFominZelevinsky05} holds for $\mathcal U$ and $\mathcal U^\prime$.

Denote by $t_0$ the initial seed of $\mathcal U$ and $I_{\rm uf}$ the set of non frozen vertices of the ice quiver $Q$. Let $k$ be the non frozen vertex of $Q$ such that $x=x_{k;t_0}$. By the starfish lemma  \cite[Corollary 1.9]{BerensteinFominZelevinsky05}, we have

\begin{eqnarray}\label{eqn:star1}
\mathcal U=\mathcal L(t_0)\bigcap \left(\bigcap\limits_{i\in I_{\rm uf}} \mathcal L(\mu_i(t_0))\right),
\end{eqnarray}
\begin{eqnarray}\label{eqn:star2}
\mathcal U^\prime=\mathcal L(t_0)\bigcap \left(\bigcap\limits_{k\neq i\in I_{\rm uf}} \mathcal L(\mu_i(t_0))\right),
\end{eqnarray}
where $\mathcal L(\mu_i(t_0))$ is the Laurent polynomial ring associated with the seed $\mu_i(t_0)$.

Since $\ca=\cu$ and $\ca'$ is a cluster localization of $\ca$ at $x$, we have that 
\[
\ca'= \ca[x^{-1}]= \cu[x^{-1}] \subseteq \cu'.
\]
It remains to show the converse inclusion $\mathcal U^\prime\subseteq \mathcal U[x^{-1}]$.

By the equality (\ref{eqn:star2}), we know that  for any $v\in \mathcal U^\prime$, there exists a positive integer $d$ such that the exponents of $x=x_{k;t_0}$ in the Laurent expansion of $vx^d$ with respect to the seed $\mu_i(t_0)$ are positive for any $k\neq i\in I_{\rm uf}$. In this case,  we have $vx^d\in\mathcal L(\mu_k(t_0))$. Then by the equality (\ref{eqn:star1}), we get  $vx^d\in\mathcal U$ and thus $v\in\mathcal U[x^{-1}]$. So we have $\mathcal U^\prime\subseteq \mathcal U[x^{-1}]$ and \[
\ca'= \ca[x^{-1}]= \cu[x^{-1}] =\cu'.
\]

\end{proof}

\subsection{Preservation of $\ca=\cu$ under deletion?} In general,
the property $\ca=\cu$ is not preserved under deletion of frozen
vertices. 
The hypotheses of the following proposition do hold for the
deletion occuring in the passage from Geiss--Leclerc--Schr\"oer's seed 
for $C_w$ to M\'enard's for $\cR_{v,w}$.

\begin{proposition} Suppose $Q'$ is obtained from $Q$ by deleting a frozen
vertex, that $Q$ and $Q'$ are of full rank and that the ice quiver $Q$ admits a reddening
sequence. Denote by $\ca$, $\cu$, $\ca'$ and $\cu'$ the cluster algebras
and the upper cluster algebras associated with $Q$ and $Q'$.
If $\ca=\cu$, then $\ca'=\cu'$.
\end{proposition}

\begin{proof} Let $\P$ and $\P'$ denote the groups of Laurent monomials
in the coefficients of $\ca$ and $\ca'$. Consider the diagram
\[
\begin{tikzcd}
\ca \ar[d, two heads] \ar[r,equal] & \cu \ar[r,hook] \ar[d,"\pi"] & \C[x_i^{\pm 1}][\P] \ar[d] \\
\ca' \ar[r,hook] & \cu' \ar[r,hook] & \C[x_i^{\pm 1}][\P']
\end{tikzcd}
\]
Let $CC$ and $CC'$ be the Caldero--Chapoton maps associated with $Q$
and $Q'$ and let $\pi: \C[x_i^{\pm 1}][\P] \to \C[x_i^{\pm 1}][\P']$ be the
specialization map. We have $\pi \circ CC = CC'$. 
By Qin's work \cite{Qin19},
we know that the image of $CC$ generates $\cu$ as a $\C[\P]$-module.
Thus, the image of $\pi\circ CC$ generates $\pi(\cu)$ as a $\C[\P']$-module.
Now the image of $\pi\circ CC$ equals that of $CC'$ and the image
of $CC'$ generates $\cu'$ as a $\C[\P']$-module since the ice quiver $Q'$ also has
a reddening sequence. It follows that $\pi(\cu)$ equals $\cu'$ and
this implies $\ca'=\cu'$.
\end{proof}


\def\cprime{$'$} \def\cprime{$'$}
\providecommand{\bysame}{\leavevmode\hbox to3em{\hrulefill}\thinspace}
\providecommand{\MR}{\relax\ifhmode\unskip\space\fi MR }
\providecommand{\MRhref}[2]{%
  \href{http://www.ams.org/mathscinet-getitem?mr=#1}{#2}
}
\providecommand{\href}[2]{#2}

\end{document}